\documentclass{amsart}
\usepackage{amssymb, amsmath}
\usepackage{graphicx}
\usepackage{mathrsfs}
\usepackage{color}
\usepackage{esint}
\newtheorem{thm}{Theorem}[section]

\newtheorem{lem}[thm]{Lemma}
\newtheorem{prop}[thm]{Proposition}
\newtheorem{prob}[thm]{Problem}
\newtheorem{ques}[thm]{Question}

\theoremstyle{definition}
\newtheorem{defn}[thm]{Definition}
\theoremstyle{remark}
\newtheorem{rem}[thm]{Remark}
\newtheorem{ex}[thm]{Example}
\numberwithin{equation}{section}

\newcommand{\R}{\mathbb R}

\newcommand{\N}{\mathbb{N}}

\newcommand{\ds}[1]{\displaystyle{ #1}}
\newcommand{\seq}[1]{\left\{ #1 \right\}_{k=0}^{\infty}}
\newcommand{\ve}{\varepsilon}
\newcommand{\LP}{\mathcal{L}-\mathcal{P}}

\newcommand{\sq}{\{\gamma_n\}_{n=0}^{\infty} }
\newcommand{\mbb}{\mathbb}
\begin{document}
\title{Hermite multiplier sequences and their associated operators}
\author{Tam\'as Forg\'acs}
\author{Andrzej Piotrowski}
\thanks{Parts of the work were completed while the first author was on sabbatical leave at the University of Hawai\textquoteleft i at Manoa, whose support he gratefully acknowledges. The authors would also like to thank their respective institutions for providing financial support for their research.}
\begin{abstract} We provide an explicit formula for the coefficient polynomials of a Hermite diagonal differential operator. The analysis of the zeros of these coefficient polynomials yields the characterization of generalized Hermite multiplier sequences which arise as Taylor coefficients of real entire functions with finitely many zeros. We extend our result to functions in $\LP$ with infinitely many zeros, under additional hypotheses.\newline {\bf MSC 30C15, 26C10}  
\end{abstract}
\maketitle

\section{Introduction}
In their much celebrated work \cite{BB}, Borcea and Br\"and\'en characterized all linear operators which preserve  geometric locations of polynomials of a single variable, thereby completing a program whose origins date back to the turn of the 20th century and the seminal paper of G. P\'olya and J. Schur \cite{ps}. P\'olya and Schur studied sequences of real numbers $\seq{\gamma_k}$ with the property that the real polynomial $\sum_{k=0}^m \gamma_k a_k x^k$ has only real zeros whenever $\sum_{k=0}^m a_kx^k$ does. They named such sequences {\it multiplier sequences} of the first kind. The classification of these objects has two of its pillars (i) the deep understanding and careful analysis of what is now known as the Laguerre-P\'olya class of real entire functions, and (ii) the Schur-Mal\'o-Szeg\H{o} composition theorem (\cite[p.7]{CCsurvey}, \cite[p. 337-340]{levin}). A natural extension to the P\'olya-Schur theory is the study of $Q$-multiplier sequences; that is, real sequences $\seq{\gamma_k}$ with the property that $\sum_{k=0}^m \gamma_k a_k q_k(x)$ has only real zeros whenever $\sum_{k=0}^m a_kq_k(x)$ does, where $Q=\seq{q_k(x)}$ is an arbitrary (fixed) basis for $\R[x]$. In a 1950 lecture \cite{turan1}, Tur\'an already pointed out the potential usefulness of such an extension to the study of Hermite multiplier sequences and their relation to the study of the Riemann $\zeta$-function. Subsequently, Tur\'an \cite{turan2}, and later Bleecker and Csordas \cite{BC} paved the way to our understanding of Hermite expansions of real entire functions and Hermite multiplier sequences, whose complete characterization was obtained in 2007 by Piotrowski \cite{andrzej}. More recently, Br\"and\'en and Ottergren \cite{bo} gave the only other known characterization involving multiplier sequences, namely those for generalized Laguerre bases.   \\
\indent Applications of Borcea and Br\"and\'en's main theorems require one to compute the action of a given linear operator $T:\R[x] \to \R[x]$ on various basis elements. One can accomplish this by identifying the unique coefficient polynomials $Q_k(x)$ in the representation
\begin{equation} \label{diffop}
T=\sum_{k=0}^{\infty} Q_k(x)D^k \qquad \left(D=\frac{d}{dx} \right),
\end{equation}
whose existence is shown in \cite{peetre} and \cite[p.32]{andrzej}.  The following two problems thus naturally arise.
\begin{prob}\label{problem1}
Given a linear operator $T:\R[x] \to \R[x]$, find explicit expressions for the coefficient polynomials $Q_k(x)$.
\end{prob}
\begin{prob}\label{problem2} Identify properties of the polynomials $Q_k(x)$ which allow one to decide directly whether or not $T$ is reality preserving. 
\end{prob}
As far as the first problem is concerned, the literature provides an answer only for operators which are diagonal (cf. Definition \ref{diagonal}) with respect to the standard basis.
\begin{defn}\label{diagonal}
Let $\mathcal{B} = \{b_n(x)\}_{n=0}^{\infty}$ be a basis for the vector space $\mathbb{R}[x]$.  A linear operator $T:\mathbb{R}[x]\rightarrow\mathbb{R}[x]$ is diagonal with respect to the basis $\mathcal{B}$ if there is a sequence of real numbers $\{\gamma_n\}_{n=0}^{\infty}$ such that 
\begin{equation}\label{bms}
T[b_n(x)] = \gamma_n b_n(x) \qquad (n=0,1,2, \dots).
\end{equation}
We shall use the terminology ``$\mathcal{B}$-diagonal'' to describe operators with this property. The reader should note that there exist linear operators on $\R[x]$ which are not diagonal with respect to any basis $\mathcal{B}$.
\end{defn}
\begin{prop}\label{sb}\emph{(\cite{andrzej} Proposition 33)} If $\sq$ is a sequence of real numbers, then the linear operator $T:\mbb{R}[x]\rightarrow\mbb{R}[x]$ defined by $T[x^n] = \gamma_n x^n$, $n \in \N_0$, has the representation 
$$
T = \sum_{k=0}^{\infty} \frac{g_k^*(-1)}{k!} x^kD^k,
$$
where
\begin{equation}\label{gstar}
g_n^*(x) = \sum_{k=0}^{n}\binom{n}{k} \gamma_k x^{n-k} \qquad (n=0,1,2,\dots)
\end{equation}
are the reversed Jensen polynomials associated to the sequence $\seq{\gamma_k}$.
\end{prop}
The Jensen polynomials associated to a real entire function $\varphi$ (or to its sequence of Taylor coefficients) have been studied extensively (\cite{CCgausslucas}, \cite{CCJensen},\cite{csvv}, \cite{db}), and can be used to give conditions on when $\varphi$ belongs to the Laguerre-P\'olya class. It is thus encouraging that the reversed Jensen polynomials should appear in the coefficient polynomials of a differential operator which is diagonal with respect to the standard basis. \\
\indent In this paper we solve Problem \ref{problem1} for operators which are diagonal with respect to a generalized Hermite basis (Theorem \ref{Tkthm}). An appealing feature of our solution is the appearance of the reversed Jensen polynomials in the formulation of the polynomials $Q_k(x)$. In addition, the coefficients of an operator diagonal with respect to the standard basis are obtained as limits of the coefficients of Hermite diagonal operators (Proposition \ref{hlimitstandard}). We also solve Problem \ref{problem2} for certain Hermite diagonal operators by demonstrating that the reality of the zeros of the coefficient polynomials can be used to determine whether a Hermite diagonal operator is reality preserving. In particular, we show that if a Hermite diagonal operator is reality preserving, then its coefficient polynomials must all have only real zeros (Theorem \ref{tkreal}). The converse of this result is false in this much generality, but becomes true if we restrict our considerations to Hermite diagonal operators associated to functions in $\LP^+$ with finitely many zeros (Theorem \ref{finitelymanyzeros}). We also obtain a converse for sequences associated to functions in $\LP^+$ with infinitely many zeros, under additional hypotheses (Theorem \ref{infinitelymanyzeros}). The penultimate section is dedicated to demonstrating that in solving Problem \ref{problem2} for Laguerre diagonal operators, one will have to use a property other than the reality of the zeros of the coefficient polynomials. The paper concludes with a list of open problems.
\section{Preliminaries}
\subsection{The Laguerre-P\'olya Class}
We begin by reviewing some notions related to real entire functions and to their membership in various function classes.
\begin{defn} A real entire function $\displaystyle{\varphi(x)=\sum_{k=0}^{\infty} \frac{\gamma_k}{k!}x^k}$ is said to belong to the Laguerre-P\'olya class, written $\varphi \in \LP$, if it can be written in the form
\[
\varphi(x)=c x^m e^{-ax^2+bx} \prod_{k=1}^{\omega} \left(1+\frac{x}{x_k} \right) e^{-x/x_k},
\]
where $b,c \in \R$, $x_k \in \R \setminus \{ 0\}$, $m$ is a non-negative integer, $a\geq 0$, $0 \leq \omega \leq \infty$ and $\displaystyle{\sum_{k=1}^{\omega} \frac{1}{x_k^2} < +\infty}$.
\end{defn}
\begin{defn}\label{LPdef} A real entire function $\displaystyle{\varphi(x)=\sum_{k=0}^{\infty} \frac{\gamma_k}{k!}x^k}$ is said to be of type I in the Laguerre-P\'olya class, written $\varphi \in \LP I$, if $\varphi(x)$ or $\varphi(-x)$  can be written in the form
\[
\varphi(x)=c x^m e^{\sigma x} \prod_{k=1}^{\omega} \left(1+\frac{x}{x_k} \right),
\]
where $c \in \R$, $m$ is a non-negative integer, $\sigma \geq 0$, $x_k >0$, $0 \leq \omega \leq \infty$ and $\displaystyle{\sum_{k=1}^{\omega} \frac{1}{x_k} < +\infty}$. If $\gamma_k \geq 0$ for $k=0,1,2.\ldots$, we write $\varphi \in \LP^+$. 
\end{defn}

With this notation, the characterization of classical multiplier sequences alluded to in the introduction can be stated as follows.
\begin{thm}\label{PStrans} \cite{ps}
A sequence $\{\gamma_k\}_{k=0}^{\infty}$ of non-negative real numbers is a multiplier sequence if and only if  $\displaystyle{\sum_{k=0}^{\infty} \frac{\gamma_k}{k!}x^k \in \LP^+}$.
\end{thm}

\subsection{Hermite Polynomials}
The following are some known facts about the classical and  generalized Hermite polynomials. For more details we refer the reader to \cite[Ch.11]{rainville} and \cite[Ch.2]{andrzej}, along with the references contained therein.
\begin{defn}
The classical Hermite polynomials are defined by 
$$
H_n(x) = (-1)^n e^{x^2} D^n e^{-x^2} \qquad (n=0,1,2,\dots),
$$
where $D$ denotes differentiation with respect to $x$.  The generalized Hermite polynomials with parameter $\alpha>0$ are defined by 
\[
\mathcal{H}_n^{(\alpha)}(x) = (-\alpha)^n \exp(x^2/2\alpha) D^n \exp(-x^2/2\alpha) \qquad (n=0,1,2,\dots).
\]
\end{defn}
The classical and generalized Hermite polynomials are related by the equation
\begin{equation}\label{clvsgen}
H_n\left(\frac{x}{\sqrt{2\alpha}}\right) = \left( \frac{2}{\alpha}\right)^{n/2} \mathcal{H}_n^{(\alpha)}(x).
\end{equation}
The generalized Hermite polynomials satisfy the limiting relation
\begin{equation}\label{Hlim}
\lim_{\alpha\to 0^+} \mathcal{H}^{(\alpha)}_n(x) = x^n \qquad (n=0,1,\ldots),
\end{equation}
the differentiation formula
\begin{equation}\label{derher}
\frac{d}{dx} \mathcal{H}_n^{(\alpha)}(x) = n  \mathcal{H}_{n-1}^{(\alpha)}(x) \qquad (\alpha>0;\, n=1,2,3,\dots),
\end{equation}
and the differential equation
\begin{equation}\label{Hdiffeq}
n\mathcal{H}_n^{(\alpha)}(x) = x D \mathcal{H}_n^{(\alpha)}(x) - \alpha D^2 \mathcal{H}_n^{(\alpha)}(x) \qquad (\alpha>0; \, n=0, 1, 2, \dots).
\end{equation}

The formula for the product of two classical Hermite polynomials (see \cite{carlitz})  
\begin{equation}\label{carlitz} 
H_n(x)H_m(x)=\sum_{i=0}^{\min \{m,n \}} 2^i i! \binom{m}{i} \binom{n}{i} H_{m+n-2 i}(x),
\end{equation}
when combined with equation (\ref{clvsgen}), yields a product formula for the generalized Hermite polynomials:
\begin{equation}\label{prodgen}
\mathcal{H}_n^{(\alpha)}(x)\mathcal{H}_m^{(\alpha)}(x)=\sum_{i=0}^{\min \{m,n \}} \alpha^i i! \binom{m}{i}\binom{n}{i}\mathcal{H}^{(\alpha)}_{m+n-2 i}(x).
\end{equation}

For the convenience of the reader, we now summarize known results on generalized Hermite multiplier sequences. As a consequence of orthogonality (see \cite[Theorem 3.3.4]{szego}), any sequence of the form
$$
\{0, 0, \dots, 0, a, b, 0, 0, \dots\} \qquad (a,b\in\mbb{R})
$$
is an $\mathcal{H}^{(\alpha)}$-multiplier sequence. These sequences are called {\it trivial $\mathcal{H}^{(\alpha)}$-multiplier sequences} and will not be considered in the remainder of the paper. 

The terms of an $\mathcal{H}^{(\alpha)}$-multiplier seqeunce are either all of the same sign, or alternate in sign. Since $\mathcal{H}^{(\alpha)}_n(x)$ is an even function for even $n$ and an odd function for odd $n$, the real sequences 
$$
\{\gamma_k\}_{k=0}^{\infty}, \quad \{-\gamma_k\}_{k=0}^{\infty}, \quad \{(-1)^k\gamma_k\}_{k=0}^{\infty}, \quad \{(-1)^{k+1}\gamma_k\}_{k=0}^{\infty}
$$
are either all $\mathcal{H}^{(\alpha)}$-multiplier sequences, or none of them are. Consequently, when characterizing $\mathcal{H}^{(\alpha)}$-multiplier sequences, it suffices to consider sequences of non-negative terms. 
\begin{thm}\cite[Theorem 152 and Lemma 161]{andrzej}\label{HMSchar}
A sequence $\{\gamma_k\}_{k=0}^{\infty}$ of non-negative terms is a (non-trivial) $\mathcal{H}^{(\alpha)}$-multiplier sequence if and only if 
$$
\varphi(x) := \sum_{k=0}^{\infty} \frac{\gamma_k}{k!}x^k \in \LP^+
$$
with $\sigma\geq 1$ (cf. Definition \ref{LPdef}). 
\end{thm}
 

\subsection{Auxiliary results} We close this section with two lemmas which will be used to determine the polynomial coefficients of a generalized Hermite diagonal operator.
\begin{lem}(\cite[Ch.2]{riordan})\label{gslem}
Let $\sq$ be a sequence of real numbers, and for $k \in\mbb{N}_0$ let $g_k^*(x)$ be defined as in (\ref{gstar}).  Then for every $n \in\mbb{N}_0$,
\[
\gamma_n=\sum_{k=0}^{n}\binom{n}{k} g_k^*(-1).
\]
\end{lem}

\begin{lem}\label{tom} For all $n\in\mbb{N}$ and all $j\in\{0,1,2,\dots,[n/2]\}$, 
\begin{equation}\label{minsum}
\sum_{k=2j}^{n}\sum_{i=0}^{M} a_{k,i} = \sum_{i=0}^{[n/2]-j}\sum_{k=i+2j}^{n-i} a_{k,i} \qquad (M = \min\{k-2j, n-k\}).
\end{equation}
\end{lem}
\begin{proof}
Suppose $n\in\mbb{N}$ and $j\in\{0,1,2,\dots,[n/2]\}$.  Note that $\min\{k-2j, n-k\} = k-2j$ if and only if $k\leq n/2 +j$ and, since all quantities involved are integers, this holds if and only if $k\leq [n/2]+j$.  Thus
\begin{eqnarray*}
\sum_{k=2j}^{n}\sum_{i=0}^{M} a_{k,i} &=& \sum_{k=2j}^{[n/2]+j} \left(\sum_{i=0}^{k-2j} a_{k,i}\right) + \sum_{k=[n/2]+j+1}^{n} \left(\sum_{i=0}^{n-k} a_{k,i} \right)\\
&=& \sum_{i=0}^{[n/2]-j}\left( \sum_{k=2j+i}^{[n/2]+j} a_{k,i}\right) + \sum_{i=0}^{n-[n/2]-j-1}  \left(\sum_{k=[n/2]+j+1}^{n-i} a_{k,i}\right)
\end{eqnarray*}  

If $n$ is odd, then $n-[n/2] = [n/2]+1$ and so 
\begin{eqnarray*}
\sum_{k=2j}^{n}\sum_{i=0}^{M} a_{k,i} &=& \sum_{i=0}^{[n/2]-j} \left(\sum_{k=2j+i}^{[n/2]+j} a_{k,i}\right) + \sum_{i=0}^{[n/2]-j} \left( \sum_{k=[n/2]+j+1}^{n-i} a_{k,i}\right)\\
&=&\sum_{i=0}^{[n/2]-j} \left(\sum_{k=2j+i}^{n-i} a_{k,i}\right).
\end{eqnarray*}

If $n$ is even, then $n-[n/2] = [n/2]$ and 
\begin{eqnarray*}
\sum_{k=2j}^{n}\sum_{i=0}^{M} a_{k,i}&=& \sum_{i=0}^{[n/2]-j}\left( \sum_{k=2j+i}^{[n/2]+j} a_{k,i}\right) + \sum_{i=0}^{[n/2]-j-1} \left( \sum_{k=[n/2]+j+1}^{n-i} a_{k,i} \right)\\
&=& \sum_{i=0}^{[n/2]-j-1} \left(\sum_{k=2j+i}^{n-i} a_{k,i} \right)+ a_{[n/2]+j,[n/2]-j}\\
&=& \sum_{i=0}^{[n/2]-j}\left( \sum_{k=2j+i}^{n-i} a_{k,i}\right).
\end{eqnarray*}

\end{proof}

\section{Main results}\label{mainresults}
We now present the central results of the paper. Theorem \ref{Tkthm} provides a closed form for the coefficient polynomials $Q_k(x)$ of an $\mathcal{H}^{(\alpha)}$-diagonal linear operator $\displaystyle{T=\sum Q_k(x)D^k}$, while Theorems \ref{tkreal}, \ref{finitelymanyzeros}  and \ref{infinitelymanyzeros} examine $\mathcal{H}^{(\alpha)}$-multiplier sequences in terms of the reality of the zeros of the polynomials $Q_k(x)$. 

\subsection{The coefficient polynomials of a Hermite diagonal linear operator}
\begin{thm} \label{Tkthm} Suppose $\alpha>0$ and $\sq$ is a sequence of real numbers.  If the linear operator $T:\mathbb{R}[x]\rightarrow\mathbb{R}[x]$ is defined by $\ds T\left[\mathcal{H}^{(\alpha)}_n(x)\right] = \gamma_n \mathcal{H}^{(\alpha)}_n(x)$ for all $n$, then $\ds T = \sum_{k=0}^{\infty} Q_k(x) D^k$, where
\begin{equation} \label{Tk}
Q_k(x) = \sum_{j=0}^{[k/2]} \frac{(-\alpha)^j}{j!(k-2j)!}g_{k-j}^*(-1) \mathcal{H}_{k-2j}^{(\alpha)}(x) \qquad (k=0,1,2,\dots),
\end{equation}
and $g_{k-j}^*(x)$ is defined as in equation $(\ref{gstar})$.
\end{thm}
\begin{rem} For ease of exposition, in our notation $Q_k(x)$ we suppress the obvious dependence on $\alpha$, unless we need to think of $\alpha$ as a variable. In this case we will write $Q^{\alpha}_k(x)$.
\end{rem}

\begin{proof}
We will show that if the polynomials $Q_k(x)$ are defined as in the theorem, then $T[\mathcal{H}_n^{(\alpha)}(x)] = \gamma_n \mathcal{H}_n^{(\alpha)}(x)$.  The fact that a differential operator representation of a linear operator is unique will then yield the desired result. 
Let 
$$
S = \sum_{k=0}^{\infty} Q_k(x) D^k \mathcal{H}_n^{(\alpha)}(x).
$$
By repeated application of the formula for the derivative of the $n^{th}$ generalized Hermite polynomial ($\ref{derher}$), we obtain
\begin{eqnarray*}
S &=& \sum_{k=0}^n \sum_{j=0}^{[k/2]} \frac{(-\alpha)^j}{j!(k-2j)!} g_{k-j}^*(-1) \mathcal{H}_{k-2j}^{(\alpha)}(x) \frac{n!}{(n-k)!} \mathcal{H}_{n-k}^{(\alpha)}(x). 
\end{eqnarray*}
Applying the formula for the product of generalized Hermite polynomials ($\ref{prodgen}$), and simplifying, leads to
\begin{eqnarray*}
S &=& \sum_{k=0}^n \sum_{j=0}^{[k/2]} \sum_{i=0}^{M} \frac{(-1)^j \alpha^{i+j} n! g_{k-j}^*(-1)}{j!i!(k-2j-i)!(n-k-i)!}  \mathcal{H}_{n-2j-2i}^{(\alpha)}(x),
\end{eqnarray*}
where $M=\min\{k-j, n-k\}$.
Changing the order of the first two sums yields
\begin{eqnarray*}
S &=& \sum_{j=0}^{[n/2]} \sum_{k=2j}^{n} \sum_{i=0}^{M} \frac{(-1)^j \alpha^{i+j} n! g_{k-j}^*(-1)}{j!i!(k-2j-i)!(n-k-i)!}  \mathcal{H}_{n-2j-2i}^{(\alpha)}(x).
\end{eqnarray*}
We now apply Lemma $\ref{tom}$, make change of variables $m=j+i$, and change the order of summation of the first two sums again to arrive at 
\begin{eqnarray*}
 S&=& \sum_{j=0}^{[n/2]} \sum_{i=0}^{[n/2]-j} \sum_{k=i+2j}^{n-i} \frac{(-1)^j \alpha^{i+j} n! g_{k-j}^*(-1)}{j!i!(k-2j-i)!(n-k-i)!}  \mathcal{H}_{n-2j-2i}^{(\alpha)}(x)\\
 &=& \sum_{j=0}^{[n/2]} \sum_{m=j}^{[n/2]} \sum_{k=m+j}^{n-m+j} \frac{(-1)^j \alpha^{m} n! g_{k-j}^*(-1)}{j!(m-j)!(k-m-j)!(n-k-m+j)!}  \mathcal{H}_{n-2m}^{(\alpha)}(x)\\
&=&\sum_{m=0}^{[n/2]} \sum_{j=0}^{m} \sum_{k=m+j}^{n-m+j} \frac{(-1)^j \alpha^{m} n! g_{k-j}^*(-1)}{j!(m-j)!(k-m-j)!(n-k-m+j)!}  \mathcal{H}_{n-2m}^{(\alpha)}(x).
\end{eqnarray*}
Substituting $p=k+j$ and changing the order of summation of the last two sums, together with the binomial theorem, produces
\begin{eqnarray*}
S&=&\sum_{m=0}^{[n/2]} \sum_{j=0}^{m} \sum_{p=m}^{n-m} \frac{(-1)^j \alpha^{m} n! g_{p}^*(-1)}{j!(m-j)!(p-m)!(n-p-m)!}  \mathcal{H}_{n-2m}^{(\alpha)}(x)\\
&=&\sum_{m=0}^{[n/2]} \sum_{p=m}^{n-m} \sum_{j=0}^{m} \frac{(-1)^j \alpha^{m} n! g_{p}^*(-1)}{j!(m-j)!(p-m)!(n-p-m)!}  \mathcal{H}_{n-2m}^{(\alpha)}(x)\\
&=&\sum_{m=0}^{[n/2]} \sum_{p=m}^{n-m} \frac{ \alpha^{m} n! g_{p}^*(-1)}{(p-m)!(n-p-m)!m!}  \sum_{j=0}^{m}(-1)^j \frac{m!}{j!(m-j)!}  \mathcal{H}_{n-2m}^{(\alpha)}(x)\\
&=&\sum_{m=0}^{[n/2]} \sum_{p=m}^{n-m} \frac{ \alpha^{m} n! g_{p}^*(-1)}{(p-m)!(n-p-m)!m!}  (1-1)^m  \mathcal{H}_{n-2m}^{(\alpha)}(x)\\
&=&\sum_{p=0}^{n} \frac{ n! g_{p}^*(-1)}{p!(n-p)!} \mathcal{H}_{n}^{(\alpha)}(x)\\
&=&\sum_{p=0}^{n} \binom{n}{p} g_{p}^*(-1) \mathcal{H}_{n}^{(\alpha)}(x)\\
&=&\gamma_n \mathcal{H}_n^{(\alpha)}(x),
\end{eqnarray*}
where we used Lemma \ref{gslem} in the last step.

\end{proof} 
As a corollary of Theorem \ref{Tkthm} we can now give a new proof for Proposition \ref{sb}.
\begin{prop}\label{hlimitstandard} If $\sq$ is a sequence of real numbers, then the linear operator $T:\mbb{R}[x]\rightarrow\mbb{R}[x]$ defined by $T[x^n] = \gamma_n x^n$ for all $n$ has the representation 

\[
T = \sum_{k=0}^{\infty} \frac{g_k^*(-1)}{k!} x^kD^k.
\]
\end{prop}
\begin{proof} Let $\alpha >0$, and consider the family of operators $T_{\alpha}:\R[x] \to \R[x]$ given by $T_{\alpha}[\mathcal{H}_n^{(\alpha)}(x)]=\gamma_n \mathcal{H}_n^{(\alpha)}(x)$. By Theorem \ref{Tkthm} if $T_{\alpha}=\sum Q^{\alpha}_k(x) D^k$, then
\[
Q^{\alpha}_k(x) = \sum_{j=0}^{[k/2]} \frac{(-\alpha)^j}{j!(k-2j)!}g_{k-j}^*(-1) H_{k-2j}^{(\alpha)}(x) \qquad (k=0,1,2,\dots).
\]
From the relation $(\ref{Hlim})$, letting $\alpha \to 0$ yields the operator $T[x^n]=\gamma_n x^n$ for all $n$, with coefficient polynomials 
\[
\lim_{\alpha\to 0^+}Q_k^\alpha(x) =\frac{g_k^*(-1)}{k!}x^k.
\]
\end{proof}
\subsection{The reality of the zeros of $Q_k(x)$}\label{realityresults}
In light of the representation (\ref{diffop}), it is a natural question to ask whether (and how) $T: \R[x] \to \R[x]$ being reality preserving is encoded in the properties of the coefficient polynomials $Q_k(x)$. The following result of Chasse (\cite[Proposition 209]{chasse}) suggests that reality preserving linear operators may distinguish themselves by having coefficient polynomials with only real zeros. 
\begin{prop}\label{finiteorder} If the operator $T: \R[x] \to \R[x]$ is reality preserving and if $T$ can be represented as a differential operator of finite order
\[
T=\sum_{k=0}^N Q_k(x)D^k, \qquad (Q_k(x) \in \R[x]),
\]
then the polynomials $Q_k(x)$ have only real zeros.
\end{prop}
Linear operators as in Proposition \ref{finiteorder} which are diagonal with respect to a simple set $\mathcal{B}=\seq{B_k(x)}$ are special in that their eigenvalues are interpolated by a polynomial. Indeed, if $\displaystyle{T=\sum_{k=0}^N Q_k(x)D^k}$ and $T[B_k(x)]=\gamma_kB_k(x)$ for all $k$, then one obtains
\begin{equation}\label{bates}
\gamma_n=\sum_{k=0}^N \binom{n}{k} Q^{(k)}_k(x).
\end{equation}
by equating leading coefficients. Hence the polynomial 
\[
p(x):=\sum_{k=0}^N \binom{x}{k} Q^{(k)}_k(x)
\]
has the property that $p(n)=\gamma_n$ for all $n \geq 0$. Since every function in $\LP^+$ interpolates an $\mathcal{H}^{(\alpha)}$-multiplier sequence (see \cite[Theorem 2.7]{BC} and \cite[Lemma 161]{andrzej}), it is useful to reformulate the result of the above discussion as follows.
\begin{prop}\label{herminfinite} If the $\mathcal{H}^{(\alpha)}$-multiplier sequence $\seq{\gamma_k}$ cannot be interpolated by a polynomial, then the differential operator representation of the associated linear operator is infinite.
\end{prop}
Chasse's result alone therefore cannot handle all operators associated with Hermite multiplier sequences. Nonetheless, as Theorem \ref{tkreal} will demonstrate, the essence of Proposition \ref{finiteorder} remains the same in the case of infinite order Hermite diagonal differential operators. We preface the statement and proof of this main result by the following lemma.

\begin{lem} \label{g*mult}
If $\seq{\gamma_k}$ is a non-trivial and non-negative $\mathcal{H}^{(\alpha)}$-multiplier sequence, then $\{g_{m+k}^*(-1)\}_{k=0}^{\infty}$ is a classical multiplier sequence for any $m\in\mbb{N}_0$.
\end{lem}

\begin{proof}
Suppose that $\seq{\gamma_k}$ is a non-trivial $\mathcal{H}^{(\alpha)}$-multiplier sequence of non-negative terms. By Theorem \ref{HMSchar}, $\ds \varphi(x) = \sum_{k=0}^{\infty} \frac{\gamma_k}{k!} x^k$ is an entire function which can be represented in the form 
\begin{equation} \label{sigmabiggerone}
c x^m e^{\sigma x} \prod_{k=0}^{\omega} \left(1+ \frac{x}{x_k}\right),
\end{equation}
where $c> 0$, $m$ is a non-negative integer, $x_k>0$, $0\leq \omega \leq \infty$, $\sum 1/x_k <\infty$ and, most importantly, $\sigma \geq 1$. 

From equation 3.6 in \cite{CCsurvey}, for any real number $t$,
$$
e^{xt} \varphi(x) = \sum_{k=0}^{\infty} g_k(1/t) \frac{(xt)^k}{k!} = \sum_{k=0}^{\infty} g_k^*(t) \frac{x^k}{k!}.
$$
Thus 
\begin{eqnarray}\label{gstardelta}
\sum_{k=0}^{\infty} g_k^*(-1) \frac{x^k}{k!} &=& e^{-x} \varphi(x) \\
&=& c x^m e^{(\sigma-1) x} \prod_{k=0}^{\omega} \left(1+ \frac{x}{x_k}\right)\in\mathcal{L-P^+}, \nonumber
 \end{eqnarray}
where the last inclusion is a consequence of $\sigma-1\geq 0$.  For $m=0$, the conclusion now follows from the characterization of classical multiplier sequences (Theorem \ref{PStrans}). Since $\seq{g_{k+m}^*(-1)}$ is a classical multiplier sequence whenever $\seq{g^*_k(-1)}$ is (see \cite[Ch. 8, Sec. 3]{levin}), the general result follows.
\end{proof}
We are now ready to extend the result of Proposition \ref{finiteorder} to infinite order differential operators which are diagonal with respect to a generalized Hermite basis.

\begin{thm}\label{tkreal} Suppose $\alpha >0$ and that $\seq{\gamma_k}$ is a non-trivial and non-negative $\mathcal{H}^{(\alpha)}$-multiplier sequence. Then the polynomials $Q_k(x)$ appearing in (\ref{Tk}) must have only real zeros for all $k \geq 0$.
\end{thm}
\begin{proof} Since $\deg Q_k(x) \leq k$ for all $k$, the zeros of the first two coefficient polynomials are always real, regardless of the sequence $\seq{\gamma_k}$. Assume for the remainder of the proof that $k \geq 2$. \\
{\bf Case 1: $k$ is even.} Using the expansion (see, e.g., \cite[p. 13]{andrzej})
\[
\mathcal{H}_k^{(\alpha)}(x)=\sum_{j=0}^{\lfloor k/2 \rfloor} \frac{k!}{2^j}\frac{(-\alpha)^j}{j!(k-2j)!}x^{k-2j}
\]
we define
\[
f(x):= x^k \mathcal{H}_k^{(\alpha)} \left(\frac{1}{x} \right)=\sum_{j=0}^{k/2} \frac{k!}{2^j}\frac{(-\alpha)^j}{j!(k-2j)!}x^{2j}.
\]
Note that $f$ is an even function with only real zeros, and hence it can be written as 
\[
f(x)=\prod_{j=1}^{k/2} (x^2-x_j), \qquad (x_j > 0, \ j=1,2,\ldots k/2).
\]
Consequently,
\[
f(\sqrt{x})=\prod_{j=1}^{k/2} (x-x_j)=\sum_{j=0}^{k/2} \frac{k!}{2^j}\frac{(-\alpha)^j}{j!(k-2j)!}x^{j} \in \mathcal{L-P}.
\]
Applying the classical multiplier sequence $\left \{2^j \right \}_{j=0}^{\infty}$ to $f(\sqrt{x})$ and dividing by $k!$ we obtain 
\[
h(x):=\sum_{j=0}^{k/2}\frac{(-\alpha)^j}{j!(k-2j)!}x^{j},
\]
whose zeros are all real. Reversing the coefficients of $h(x)$ and applying the classical multiplier sequence $\left \{g_{k/2+j}^*(-1)\right \}_{j=0}^{\infty}$ (see Lemma \ref{g*mult}) we arrive at
\[
\widetilde{h}(x)=\sum_{j=0}^{k/2} \frac{(-\alpha)^j}{j!(k-2j)!}g^*_{k-j}(-1)x^{k/2-j},
\]
which is a polynomial with only real zeros. Note further, that the zeros of $\widetilde{h}(x)$ must also be {\it positive}, in light of Lemma \ref{g*mult}, since its coefficients are alternating in sign. We conclude that 
\[
\widetilde{h}(x^2)=\sum_{j=0}^{k/2} \frac{(-\alpha)^j}{j!(k-2j)!}g^*_{k-j}(-1)x^{k-2j}
\]
also has only real zeros. Since the linear operator $x^k \to \mathcal{H}^{(\alpha)}_k$, $k \in \N_0$,  is reality preserving (see \cite[Example 30 and Theorem 38]{andrzej}), the polynomial
\[
Q_k(x)=\sum_{j=0}^{\lfloor k/2 \rfloor} \frac{(-\alpha)^j}{j!(k-2j)!}g_{k-j}^*(-1) \mathcal{H}_{k-2j}^{(\alpha)}(x)
\]
has only real zeros. \\
{\bf Case 2: $k$ is odd.} In this case we define $f$ slightly differently by
\[
x \cdot f(x):= x^{k+1} \mathcal{H}_k^{(\alpha)} \left(\frac{1}{x} \right)=x \cdot \sum_{j=0}^{(k-1)/2} \frac{k!}{2^j}\frac{(-\alpha)^j}{j!(k-2j)!}x^{2j}.
\]
The steps now are identical to those in the even case: we apply the classical multiplier sequence $\left \{2^{j+1}\right\}_{j=0}^{\infty}$ to $x \cdot f(\sqrt{x})$, and divide by $k!$ to obtain 
\[
x \cdot h(x):=x \cdot \sum_{j=0}^{(k-1)/2}\frac{(-\alpha)^j}{j!(k-2j)!}x^{j},
\]
whose zeros are all real. Reversing the coefficients of $h(x)$ and applying the classical multiplier sequence $\left \{g_{(k-1)/2+j+1}^*(-1)\right\}_{j=0}^{\infty}$ leads to 
\[
x \cdot \widetilde{h}(x)=x \cdot \sum_{j=0}^{(k-1)/2} \frac{(-\alpha)^j}{j!(k-2j)!}g^*_{k-j}(-1)x^{(k-1)/2-j},
\]
a polynomial with only real positive zeros. Hence
\[
x \cdot \widetilde{h}(x^2)=\sum_{j=0}^{(k-1)/2} \frac{(-\alpha)^j}{j!(k-2j)!}g^*_{k-j}(-1)x^{k-2j},
\]
and subsequently 
\[
Q_k(x)=\sum_{j=0}^{\lfloor k/2 \rfloor} \frac{(-\alpha)^j}{j!(k-2j)!}g_{k-j}^*(-1) \mathcal{H}_{k-2j}^{(\alpha)}(x)
\]
has only real zeros. The proof is complete.
\end{proof}




\subsection{The converse of Theorem \ref{tkreal}}
As seen in Theorem \ref{HMSchar}, whether or not a sequence of non-negative numbers $\seq{\gamma_k}$ is a generalized Hermite multiplier sequence depends entirely on whether $\sigma\geq 1$ in Definition \ref{LPdef}. Thus it is natural to expect that $\sigma$ should play a role in determining whether or not the coefficient polynomials $Q_k(x)$ have only real zeros. Conversely, the polynomials $Q_k(x)$ having only real zeros should imply that $\sigma \geq 1$. As it turns out, the latter implication is not always true. Consider for example the $\mathcal{H}^{(\alpha)}$-diagonal linear operator
\[
T=a+ x D - \alpha D^2 \qquad (a \in \R),
\]
which, by equation (\ref{Hdiffeq}), represents the sequence $\seq{k+a}$. Despite the fact that all coefficients of $T$ have only real zeros, the sequence is an $\mathcal{H}^{(\alpha)}$-multiplier sequence if and only if $a \geq 0$. We are thus led to consider only those sequences which are Taylor coefficients of functions in $\LP^+$  in order to establish a converse of Theorem \ref{tkreal}. We continue our investigation with the following lemma.

\begin{lem}\label{turanish} Assume the setup of Theorem \ref{Tkthm} and
suppose $k \geq 2$. If $Q_k(x)$ has only real zeros, then $\ds [g_{k}^{*}(-1)]^2 + 2 g_{k}^{*}(-1) g_{k-1}^{*}(-1)\geq 0$.
\end{lem}

\begin{proof}

\begin{eqnarray*}
Q_k(x) &=& \sum_{j=0}^{[k/2]} \frac{(-\alpha)^j}{j!(k-2j)!}g_{k-j}^*(-1) \mathcal{H}_{k-2j}^{(\alpha)}(x) \\
&=& \frac{1}{k!} g_k^*(-1) \mathcal{H}_k^{(\alpha)}(x) + \frac{-\alpha}{(k-2)!} g_{k-1}^*(-1) \mathcal{H}_{k-2}^{(\alpha)}(x) + \text{lower order terms}\\
&=& \frac{1}{k!} g_k^*(-1) \left( x^k - \alpha \frac{ k!}{2 (k-2)!} x^{k-2} + \text{lower order terms}  \right) \\
&+& \frac{-\alpha}{(k-2)!} g_{k-1}^*(-1)\left(x^{k-2}+\text{lower order terms}\right).
\end{eqnarray*}
If $Q_k(x)$ has only real zeros, then the coefficients of $x^k$ and $x^{k-2}$ must have opposite signs unless one of them is zero.  Therefore
$$
\frac{1}{k!} g_k^*(-1) \frac{-\alpha}{(k-2)!}\left[ \frac{1}{2} g_{k}^*(-1) + g_{k-1}^*(-1)   \right] \leq 0
$$
and the result follows.
\end{proof}

\begin{thm} \label{finitelymanyzeros} Suppose that $\alpha >0$ and 
\[
\varphi(x)=\sum_{k=0}^{\infty} \frac{\gamma_k}{k!}x^k= c x^m e^{\sigma x} \prod_{k=1}^{N}\left(1+\frac{x}{x_k} \right) \in \LP^+
\]
for some $N \in \N \cup \{ 0 \}$. If the polynomials $Q_k(x)$ given in equation (\ref{Tk}) have only real zeros for $k \geq 2$, then $\seq{\gamma_k}$ is an $\mathcal{H}^{(\alpha)}$-multiplier sequence.
\end{thm}
\begin{proof} Suppose that the polynomials $Q_k(x)$ given in equation (\ref{Tk}) have only real zeros for $k \geq 2$ and that $\varphi(x)$ is as in the statement of the theorem. Write 
\[
\prod_{k=1}^{N}\left(1+\frac{x}{x_k} \right)=\sum_{k=0}^N a_k x^k.
\]
Calculating $\varphi^{(k)}(0)$ directly gives $\gamma_j=0$ for $j=0,1,\ldots,m-1$ and
\begin{equation}\label{gstars}
\gamma_{k+m}=(k+m)! \sum_{j=0}^{\min\{k,N\}} a_{j}\frac{\sigma^{k-j}}{(k-j)!}, \qquad (k=0,1,2,\ldots).
\end{equation}
Combining equations (\ref{gstardelta}) and (\ref{gstars}) for $k>N$ we obtain
\begin{equation}\label{firstconnection}
g_{k+m}^*(-1)=c (k+m)! \sum_{j=0}^N a_{j} \frac{(\sigma-1)^{k-j}}{(k-j)!}.
\end{equation} 
From equation (\ref{firstconnection}) we see that at most finitely many of the coefficients $g_k^*(-1)$ can vanish.
Consequently, if $k \gg N$, then
\begin{eqnarray*}
\frac{g_{k+1+m}^*(-1)}{g_{k+m}^*(-1)}&=&(k+1+m) \frac{\sum_{j=0}^N a_{j} \frac{(\sigma-1)^{k+1-j}}{(k+1-j)!}}{\sum_{j=0}^N a_{j} \frac{(\sigma-1)^{k-j}}{(k-j)!}} \\
&=&\frac{k+1+m}{k+1-N}(\sigma-1) \frac{a_{N}+\mathcal{O}\left( \frac{1}{(k+2-N)!}\right)}{a_{N}+\mathcal{O}\left( \frac{1}{(k+1-N)!}\right)}.
\end{eqnarray*}
From here we conclude that 
\[
\frac{g_{k+1+m}^*(-1)}{g_{k+m}^*(-1)} \to (\sigma-1) \qquad \textrm{as} \quad k \to \infty.
\]
Assume now that $\seq{\gamma_k}$ is not an $\mathcal{H}^{(\alpha)}$-multiplier sequence. Then $0 \leq \sigma <1$ in the Hadamard factorization of $\varphi$, and hence $-1 \leq \sigma-1<0$. Consequently,
\begin{eqnarray*}
g_{k+1+m}^*(-1)^2+2 g_{k+1+m}^*(-1)g_{k+m}^*(-1)&=& g_{k+1+m}^*(-1)\left[g_{k+1+m}^*(-1)+2g_{k+m}^*(-1) \right]\\
&\to& g_{k+1+m}^*(-1)g_{k+m}^*(-1)\left[ 2+(\sigma-1) \right] <0.
\end{eqnarray*}
Lemma \ref{turanish} now implies that the polynomials $Q_k(x)$ must  have non-real zeros for $k \gg 1$, a contradiction.
\end{proof}
The above proof gives the first insight as to how the magnitude of $\sigma$ may influence the reality of the zeros of $Q_k(x)$ through the quantities $g_k^*(-1)$. We further offer the following examples as illustration of Theorem \ref{finitelymanyzeros} as well as motivation for Theorem \ref{infinitelymanyzeros}.
\begin{ex}\label{simpleexample} Consider the function
\[
\sum_{k=0}^{\infty} \frac{4(k+1)^2-8(k+1)+5}{2^kk!} x^k=e^{\frac{1}{2}x} (1+x)^2 \in \LP^+.
\] 
It has one zero ($x=-1)$ with multiplicity 2, and hence we have
\[
a_{0}=1, \qquad a_{1}=2, \qquad  a_{2}=1.
\]
The first few ratios $\displaystyle{\frac{g_{k+1}^*(-1)}{g_{k}^*(-1)}}$ are given in the table below, with $\displaystyle{\lim_{n \to \infty} \frac{g_{n+1}^*(-1)}{g_n^*(-1)}=-\frac{1}{2}}$. 
\begin{table}[htdp]
\caption{$g_{k+1}^*(-1)/g_{k}^*(-1)$ for Example \ref{simpleexample}}
\begin{center}
\begin{tabular}{c|cccccccc}
$k$ & 1 & 2 & 3 & 4 & 5 & 6 & 7    \\
\hline
\hline
$g_{k+1}^*(-1)/g_{k}^*(-1)$ & 3/2 & 1/6 & -13/2 &-33/26 & -61/66 & -97/122 & -141/194
\end{tabular}
\end{center}
\end{table}%

\end{ex}
\begin{ex} The classical multiplier sequence $\displaystyle{\seq{ \frac{1}{k!}}}$ represents the Taylor coefficients of the Bessel function of the first kind with index zero (see for example\cite[p.108]{rainville}):
\[
J_0(2 \sqrt{x})=\sum_{j=0}^{\infty} \frac{x^k}{k!k!}.
\]
The corresponding sequence $\displaystyle{\seq{g_k^*(-1)}=\seq{\sum_{j=0}^k \binom{k}{j}\frac{1}{j!}(-1)^{k-j}}}$ starts numerically as
\[
1,0,-\frac{1}{2}, \frac{2}{3}, -\frac{5}{8}, \frac{7}{15}, -\frac{37}{144}, \frac{17}{420}, \ldots
\]
Computing the values of $(g_{k}^*(-1))^2+2g_{k}^*(-1) g_{k-1}^*(-1)$ for $k=1,2,3,4,5$ we get 
\[
0,1,\boxed{-\frac{2}{9},-\frac{85}{192}, -\frac{329}{900}}.
\]
We conclude that $Q_3(x)$ (as well as $Q_4(x)$ and $Q_5(x)$) has non-real zeros. Indeed,
\begin{eqnarray*}
Q_3(x)&=&\frac{g_3^*(-1)}{3!}\mathcal{H}_3^{(\alpha)}(x)-\alpha g_2^*(-1)\mathcal{H}_1^{(\alpha)}\\
&=&\frac{(x^2+6 \alpha)x}{18}
\end{eqnarray*}
from which the conclusion follows easily, since $\alpha >0$.
\end{ex}
\begin{ex}
 Finally, we consider the sequence $\displaystyle{\seq{\frac{r^k}{k!}}}$. A calculation shows that in this case
 \[
 Q_2(x)=\frac{1}{2}\left(1-2 r+\frac{r^2}{2} \right)x^2-\frac{\alpha}{2}\left(\frac{r^2}{3}-1 \right),
 \]
 which has non-real zeros if $\displaystyle{r \in \left[0, \frac{1}{\sqrt{3}}\right) \cup \left(2-\sqrt{2},1\right]}$. If $\displaystyle{ \frac{1}{\sqrt{3}} \leq r \leq 2-\sqrt{2}}$, then $Q_4(x)$ has non-real zeros.  
\end{ex}
Before we can state a converse for Theorem \ref{tkreal} for real entire functions with infinitely many zeros, we need two preliminary results.
\begin{defn} \label{kp} Given a sequence of real numbers $\seq{\gamma_k}$ and $p \in \N \cup \{0 \}$ we define
\[
g_{k,p}^*(-1):=\sum_{n=0}^k \binom{k}{n} \gamma_{n+p}(-1)^{k-n},
\]
and
\[
Q_{k,p}(x):= \sum_{j=0}^{[k/2]} \frac{(-\alpha)^j}{j!(k-2j)!}g_{k-j,p}^*(-1) \mathcal{H}_{k-2j}^{(\alpha)}(x).
\]
Note that setting $p=0$ returns the definitions of $g_k^*(-1)$ and $Q_k(x)$ given in equations (\ref{gstar}) and (\ref{Tk}) respectively.
\end{defn}
\begin{lem}\label{convergentgamma} If $\seq{\gamma_k}$ is convergent, then for any fixed $k \geq 2$, $\lim_{p \to \infty} |g_{k,p}^*(-1)|=0$. 
\end{lem}
\begin{proof} Fix $k \geq 2$ and that suppose  $\gamma_j \to \gamma$. Note that
\[
0=\gamma (1-1)^k=\sum_{n=0}^k \binom{n}{k} \gamma (-1)^{n-k}.
\]
Let $\ve>0$, and select $P \in \N$, such that $p \geq P$ implies $|\gamma_p-\gamma| < \ve/2^k$. Then for $p \geq P$ we have
\begin{eqnarray*}
|g_{k,p}^*(-1)-0|&=& \left|\sum_{n=0}^k \binom{k}{n} (\gamma_{n+p}-\gamma) (-1)^{k-n} \right|\\
&\leq& \sum_{n=0}^k \binom{k}{n}|\gamma_{n+p}-\gamma| \\
&<& \frac{\ve}{2^k}\sum_{n=0}^k \binom{k}{n}\\
&=&\ve,
\end{eqnarray*}
from which the conclusion follows.
\end{proof}
\begin{lem}\label{shifty} For all $k \geq 2$ and $p \geq 0$ the following equality holds: 
\[
g_{k,p}^*(-1)+g_{k+1,p}^*(-1)=g_{k,p+1}^*(-1).
\]
\end{lem}
\begin{proof}We compute directly
\begin{eqnarray*}
g_{k,p}^*(-1)+g_{k+1,p}^*(-1)&=&\sum_{n=0}^k \binom{k}{n} \gamma_n+p (-1)^{k-n}+\sum_{n=0}^{k+1} \binom{k+1}{n} \gamma_n+p (-1)^{k+1-n}\\
&=&\gamma_{k+1+p}+\sum_{n=0}^k \left( \binom{k+1}{n}-\binom{k}{n} \right)\gamma_{n+p} (-1)^{k+1-n}\\
&=&\gamma_{k+1+p}+\sum_{n=1}^k \binom{k}{n-1}\gamma_{n+p} (-1)^{k+1-n}\\
&=&\gamma_{k+1+p}+\sum_{n=0}^{k-1} \binom{k}{n}\gamma_{n+p+1} (-1)^{k-n}\\
&=&\sum_{n=0}^{k} \binom{k}{n}\gamma_{n+p+1} (-1)^{k-n}\\
&=&g_{k,p+1}^*(-1).
\end{eqnarray*}
\end{proof}
We are now in position to state and prove a converse of Theorem \ref{tkreal} in the case of a real entire function with infinitely many zeros.
\begin{thm} \label{infinitelymanyzeros} Suppose that
\[
\varphi(x)=\sum_{k=0}^{\infty} \frac{\gamma_k}{k!}x^k=cx^me^{\sigma x} \prod_{k=1}^{\infty} \left(1+\frac{x}{x_k} \right) \in \LP^+,
\]
with $c \in \R$, $m \geq 0$, $x_k >0$ and $\sum \frac{1}{x_k} <+\infty$. If $Q_{k,p}(x)$, as in definition \ref{kp}, has only real zeros for all $k \geq 2$ and all $p \geq 0$, then $\seq{\gamma_k}$ is a Hermite multiplier sequence.   
\end{thm}
\begin{proof} In search for a contradiction, assume that $Q_{k,p}(x) \in \LP$ for all $k \geq 2$ and $p \geq 0$, but that the sequence $\seq{\gamma_k}$ of non-negative terms is not an Hermite multiplier sequence. Then by Theorem 4.8 in \cite[p. 427]{CCgausslucas}, the sequence $\seq{\gamma_k}$ is bounded and eventually monotone, hence convergent. By Lemma \ref{convergentgamma} we conclude that for any fixed $k\geq 2$, $\displaystyle{\lim_{p \to \infty} |g_{k,p}^*(-1)|=0}$. Since $Q_{k,p}(x) \in \LP$ for all $k \geq 2$ and $p \geq 0$, Lemma \ref{turanish} gives
\[
(\dag) \quad (2 g_{k,p}^*(-1) +g_{k+1,p}^*(-1))g_{k+1,p}^*(-1) \geq 0 \qquad k \geq 1, p \geq 0.
\]
Using Lemma \ref{shifty} we may rewrite $(\dag)$ as
\[
(g_{k,p+1}^*(-1))^2-(g_{k,p}^*(-1))^2 \geq 0 \qquad k \geq 2, p \geq 0,
\]
or equivalently, $\displaystyle{|g_{k,p+1}^*(-1)| \geq |g_{k,p}^*(-1)|}$. Fixing $k$ and letting $p \to \infty$ produces a monotone increasing sequence of non-negative numbers whose limit is zero. We conclude that $g_k^*(-1)=0$ for all $k \geq 2$. On the other hand, $e^{-x}\varphi(x)$ has infinitely many zeros, and its Taylor coefficients are given by the sequence $\seq{g_{k}^*(-1)}$. We thus see that infinitely many of the $g_k^*(-1)$ must be non-zero, and we have reached a contradiction.
\end{proof}

\section{A note on Laguerre multiplier sequences and their associated operators}\label{laguerresection}
In \cite{bo}, the authors demonstrate (Theorem 1.1) that the linear operator corresponding to any Laguerre multiplier sequence is a finite order differential operator. It follows from Proposition \ref{finiteorder}  that the coefficient polynomials of any operator associated to a Laguerre multiplier sequence have to have only real zeros. In the spirit of the current paper, the following question arises.
\begin{ques} If $\seq{\gamma_k}$ is a classical multiplier sequence, $T[L_n^{(\alpha)}(x)]=\gamma_n L_n^{(\alpha)}(x)$ for all $n$, and $\displaystyle{
T=\sum_{k=0}^n Q_k(x)D^k}$ where $Q_k(x) \in \mathcal{L-P}$ for $k=0,1,\ldots,n$, must it follow that $\seq{\gamma_k}$ is a $L^{(\alpha)}$-multiplier sequence?
\end{ques}
Perhaps somewhat surprisingly, the answer to this question is no, substantiated by the following simple considerations.  The Laguerre diagonal operator corresponding to the sequence $\seq{k+a}$ has the differential operator representation
\[
T=a+(x-\alpha-1)D-xD^2.
\]
It is apparent that all coefficient polynomials of $T$ have only real zeros regardless of the value of $a$ or $\alpha$. 
While $\seq{k+a}$ is a classical multiplier sequence for every $a \geq 0$, it is an $L^{(\alpha)}$-multiplier sequence if and only if $0\leq a \leq \alpha+1$ (see \cite{tomandrzej}).
Thus, even though every generalized Laguerre multiplier sequence is an $\mathcal{H}^{(\alpha)}$-multiplier sequence, Theorem \ref{finitelymanyzeros} has no counterpart in the generalized Laguerre setting. 

\section{Open problems}\label{open}
In light of the penultimate section, our main results do not necessarily have counterparts for bases other than the generalized Hermite bases. We thus pose the following problem:
\begin{prob} Characterize all bases (or simple sets) $\mathcal{B}=\seq{b_k(x)}$ with the following property: if $T[b_k(x)]=\gamma_kb_k(x)$, and $\displaystyle{T=\sum Q_k(x)D^k}$, then $\seq{\gamma_k}$ is a $\mathcal{B}$-multiplier sequence if and only if the polynomials $Q_k(x)$ have only real zeros.
\end{prob}
The standard basis and the generalized Laguerre bases do not have this property.

\bigskip

We believe that the conclusion of Theorem \ref{infinitelymanyzeros} holds even in the case when one only assumes that $Q_k(x) \in \LP$ for $k \geq 2$. No methods known to us at this time yielded a proof of this fact, hence we pose
\begin{prob} Prove that if $T=\sum Q_k(x)D^k$ is a Hermite diagonal operator associated to a classical multiplier sequence, and $Q_k(x)\in \LP$ for all $k$, then $T$ is reality preserving. 
\end{prob}
The techniques used in the proof of Theorem \ref{finitelymanyzeros} should be considered for a possible extension to functions in $\mathcal{L-P}^+$ with infinitely many zeros. The presence of infinitely many zeros requires a subtle and careful analysis, for in this case the limit $\displaystyle{\lim_{k \to \infty} g^*_{k+1+m}(-1)/g_{k+m}^*(-1)}$ need not exist. The graph below shows the first two hundred values of the above quotient for the function $\displaystyle{\varphi(x)=e^{x/2} \cosh(\sqrt{2x})}$. 

\begin{figure}[htbp]
\begin{center}
\includegraphics[width=3 in]{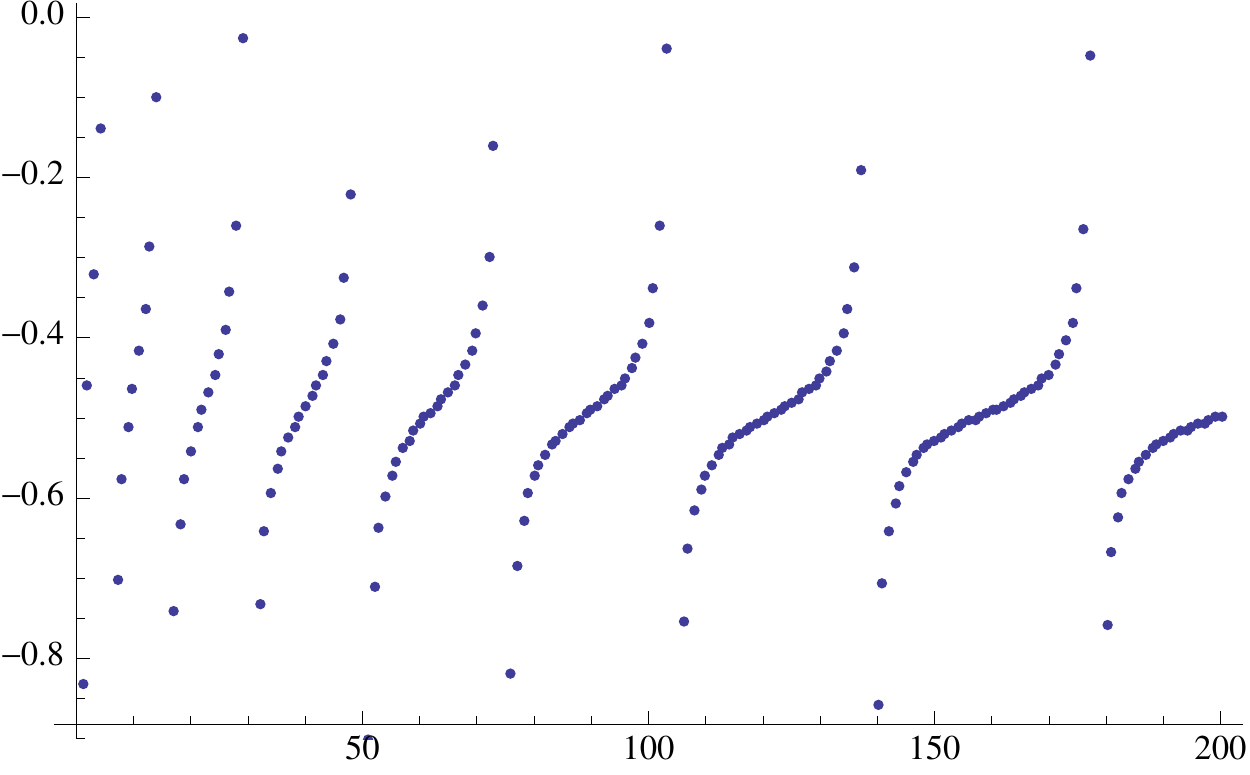}
\caption{The quotient $g_{k+1}^*(-1)/g_k^*(-1)$ for the function $e^{x/2} \cosh (\sqrt{2x})$}
\label{default}
\end{center}
\end{figure}
Looking at a histogram of these values we are led to believe that probabilistic methods may establish that most (and in fact infinitely many) quotients are near $\sigma-1$, and hence there will be a $Q_k(x)$ with non-real zeros. 

\begin{figure}[htbp]
\begin{center}
\includegraphics[width=3 in]{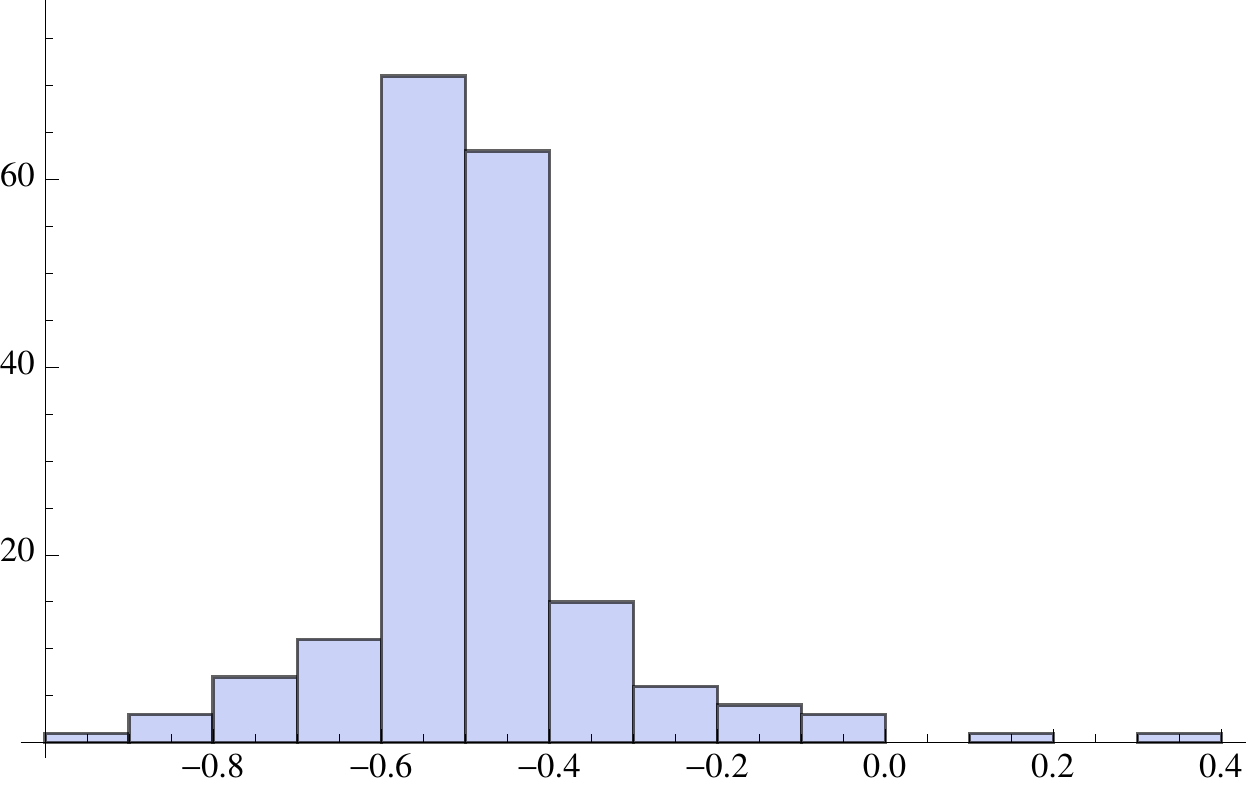}
\caption{A histogram of the first 200 values of the quotient $g_{k+1}^*(-1)/g_k^*(-1)$ for the function $e^{x/2} \cosh (\sqrt{2x})$}
\label{default2}
\end{center}
\end{figure}

In any case, we would like to find the answer to the following.
\begin{prob} Suppose that
\[
\varphi(x)=c x^m e^{\sigma x} \prod_{k=1}^{\infty} \left(1+\frac{x}{x_k} \right) \in \mathcal{L-P}^+
\]
with $0 \leq \sigma < 1$. Is it true that $\sigma-1$ is a subsequential limit of the sequence $\seq{g_{k+1}^*(-1)/g_k^*(-1)}$? More generally, what can we say about the connection between the type of a real entire function of order 1, and the sequence of quotients of its consecutive Taylor coefficients?
\end{prob}
\bigskip

{\bf Acknowledgement:} We would like to thank George Csordas for many stimulating discussions, for some illuminating examples, and for his enduring encouragement during the completion of this work. We also express our gratitude to the participants of the University of Hawai\textquoteleft i 2013 Fall graduate complex analysis seminar for bringing the expression in equation (\ref{bates}) to our attention.

\end{document}